\documentclass{birkau}

\usepackage{amsmath,amssymb,url,color}
\usepackage[shortlabels]{enumitem} 
\usepackage{savesym}
\savesymbol{bibsection}
\usepackage{amsrefs}

\numberwithin{equation}{section}

\theoremstyle{plain}
\newtheorem{theorem}{Theorem}[section]
\newtheorem{lemma}[theorem]{Lemma}
\newtheorem{proposition}[theorem]{Proposition}
\newtheorem{corollary}[theorem]{Corollary}
\theoremstyle{definition}
\newtheorem{example}[theorem]{Example} 
\theoremstyle{definition}

\newtheorem{remark}[theorem]{Remark}
\newtheorem*{fact}{Fact}

\newenvironment{rlist}
   {\begin{list}{}{\setlength{\labelsep}{0.2cm}
\setlength{\itemsep}{.125cm}
\setlength{\topsep}{0cm}
                   \setlength{\labelwidth}{0.7cm}
                      \setlength{\leftmargin}{0.75cm}}}  
   {\end{list}}
\newcommand{\cop}[1]{\mathbb{#1}}

\newcommand{\lra}{\leftrightarrow}
\newcommand{\lan}{\langle}
\newcommand{\ran}{\rangle}

\newcommand{\mt}{\land}
\newcommand{\jn}{\lor}

\newcommand{\eq}{\approx}

\newcommand{\m}[1]{{\bf {#1} }}
\newcommand{\f}{\ensuremath{\varphi}}
\newcommand{\ps}{\ensuremath{\psi}}

\newcommand{\si}{\ensuremath{\sigma}}
\newcommand{\de}{\ensuremath{\delta}}

\newcommand{\cg}[1]{{\rm Cg}_{_{#1}}}

\newcommand{\cls}[1]{\mathcal{#1}}
\newcommand{\mdl}[1]{\models_{#1}}
\newcommand{\der}[1]{\vdash_{#1}}
\newcommand{\notmdl}[1]{\not\models_{#1}}
\newcommand{\lgc}[1]{\ensuremath{{\sf #1}}}
\newcommand{\adm}[1]{\mathrel{\makebox{\raisebox{.4ex}{\scriptsize $\mid$}\raisebox{.28ex}{\footnotesize $\! \sim$}}_{\lgc{#1}}}}
\newcommand{\De}{\Delta}
\newcommand{\Ga}{\Gamma}
\newcommand{\Si}{\Sigma}
\newcommand{\The}{\Theta}

\newcommand{\F}{\m{F}}

\newcommand{\Tm}{\m{Tm}}
\newcommand{\Eqc}{\mathrm{Eq}}
\newcommand{\Tmc}{\mathrm{Tm}}
\newcommand{\con}[1]{\textup{Con}({#1})}   
\newcommand{\lang}{\mathcal{L}}

\newcommand{\xbar}{\overline{x}}
\newcommand{\ybar}{\overline{y}}
\newcommand{\zbar}{\overline{z}}

\newcommand{\eps}{\ensuremath{\varepsilon}}
\newcommand{\N}{\mathbb{N}}
\newcommand{\V}{\mathcal{V}}

\newcommand{\K}{\mathcal{K}}

\newcommand{\Lat}{\mathcal{L}\mathit{at}}
\newcommand{\DLat}{\mathcal{DL}\mathit{at}}
\newcommand{\DiagpA}{{\rm Diag}^+(\m{A})}

\begin{document}


\title[Deciding dependence in logic and algebra]{\LARGE Deciding dependence in logic and algebra}

\author[G. Metcalfe]{George Metcalfe}
\address{Mathematical Institute, University of Bern\\
Sidlerstrasse 5, 3012 Bern, Switzerland}
\email{george.metcalfe@math.unibe.ch}

\author[N. Tokuda]{Naomi Tokuda}
\address{Mathematical Institute, University of Bern\\
Sidlerstrasse 5, 3012 Bern, Switzerland}
\email{naomi.tokuda@math.unibe.ch}

\thanks{This research was supported by Swiss National Science Foundation grant 200021$\_$165850.}

\dedicatory{Dedicated to Dick de~Jongh}



\begin{abstract}
We introduce a universal algebraic generalization of de Jongh's notion of dependence for formulas of  intuitionistic propositional logic, relating it to a notion of dependence defined by Marczewski for elements of an algebraic structure. Following ideas of de~Jongh and Chagrova, we show how constructive proofs of (weak forms of) uniform interpolation can be used to decide dependence for varieties of abelian $\ell$-groups, MV-algebras, semigroups,  and modal algebras. We also consider minimal provability results for dependence, obtaining in particular a complete description and decidability of dependence for the variety of lattices.
\end{abstract}

\maketitle


\section{Introduction}\label{sec:introduction}

In~\cite{dJC95} de~Jongh and Chagrova introduced an intriguing notion of dependence of formulas for intuitionistic propositional logic $\lgc{IPC}$ as an analogue of the usual notion of dependence of vectors in linear algebra.\footnote{This notion was previously defined by de~Jongh in~\cite{deJ82} using different terminology and appears also in  Lemmon's textbook {\em Beginning Logic}~\cite{Lem65} in the context of comparing two formulas of classical propositional logic. For a detailed history and comparison of this and other notions of dependence in logic, we refer to~\cite{Hum20}. Details of the seemingly unrelated but fascinating field of {\em dependence logic} may be found in~\cites{Vaa07,YV16}.} Formulas $\f_1,\dots,\f_n$ are said to be {\em $\lgc{IPC}$-dependent} if there exists a formula $\ps(p_1,\dots,p_n)$ such that $\der{\lgc{IPC}} \ps(\f_1,\dots,\f_n)$, but  $\not\der{\lgc{IPC}} \ps(p_1,\dots,p_n)$; otherwise, $\f_1,\dots,\f_n$ are said to be  {\em $\lgc{IPC}$-independent}. In~\cite{deJ82} it is shown that a single formula $\f$ is $\lgc{IPC}$-dependent if, and only if, either $\der{\lgc{IPC}} \lnot\lnot \f \to \f$ or $\der{\lgc{IPC}} \lnot\lnot \f$, and~\cite{dJC95} provides a ``reasonably simple'' sequence of formulas $(\ps_i(p_1,p_2))_{i\in\N}$ such that any two formulas  $\f_1,\f_2$ are dependent if, and only if, $\der{\lgc{IPC}} \ps_i(\f_1,\f_2)$ for some $i\in\mathbb{N}$. Notably, it is also proved in~\cite{dJC95} that checking the $\lgc{IPC}$-dependence of any finite number of formulas is decidable, making ingenious use of Pitts' constructive proof of uniform interpolation for $\lgc{IPC}$~\cite{Pit92}.

In this paper, we extend the ideas of~\cite{dJC95} to a general universal algebraic setting (recalled in Section~\ref{sec:consequence}), defining the  $\V$-dependence of terms $t_1,\dots,t_n$ for an arbitrary variety (equational class)~$\V$. In Section~\ref{sec:algebraic}, we show that $t_1,\dots,t_n$ are $\V$-independent if, and only if, a certain homomorphism between finitely generated  $\V$-free algebras is an embedding, or, equivalently, a certain finitely generated subalgebra of a $\V$-free algebra is $\V$-free over $n$ generators. It then follows that $\V$-dependence is a special case of Marczewski's notion of dependence for elements of an algebraic structure~\cites{Mar58,Mar59,Mar61}, and a converse is obtained by relativizing $\V$-dependence to a given set of equations. Marczewski-dependence has been studied extensively in universal algebra (see~\cite{Gra08}*{Chapter 5}), leading in particular to the introduction of $v^*$-algebras~\cite{Nar61},  later rediscovered by semigroup theorists in the guise of independence algebras~\cite{Gou95}.

In Section~\ref{sec:deciding}, we turn our attention to the problem of deciding $\V$-dependence. Generalizing the proof strategy of~\cite{dJC95}, we show that $\V$-dependence is decidable (even relative to a finite set of equations) if there is a constructive proof that $\V$ is {\em coherent}, i.e., that every finitely generated subalgebra of a finitely presented algebra in $\V$ is finitely presented. The property of coherence originated in sheaf theory and has been studied widely in algebra (see, e.g.,~\cites{Sou70,Sch83,Gou92}), and from a more general model-theoretic perspective in~\cites{Whe76,Whe78}. In~\cite{KM18} it was proved that coherence and deductive interpolation are jointly equivalent to the property of right uniform deductive interpolation considered in~\cite{vGMT17}. 

Constructive proofs of coherence are implicit in uniform interpolation proofs for, e.g., $\lgc{IPC}$~\cite{Pit92}, the modal logics $\lgc{GL}$ and $\lgc{S4Grz}$~\cite{Bil07}, and the varieties of abelian $\ell$-groups and MV-algebras~\cites{MMT14,vGMT17}. However, coherence is a quite rare property for non-locally finite varieties, at least those corresponding to modal, substructural, and other non-classical logics~\cites{KM18,KM18a}. We therefore also consider a weaker condition for deciding $\V$-dependence that requires only a constructive proof that finitely generated $\V$-free algebras are coherent, and is satisfied, for example, by the non-coherent variety of groups. Even when such a constructive proof is not available, however, there may exist other methods for deciding $\V$-dependence. Notably, the dependence problem for the variety of semigroups corresponds precisely to checking whether a finite set of words is a code, solved by the famous Sardinas-Patterson algorithm~\cite{SP53}, and the dependence problem for the variety of modal algebras can be decided using bisimulation-based methods described in~\cite{LW11} for calculating uniform interpolants (when they exist) for the description logic $\lgc{ALC}$.

In Section~\ref{sec:minimal}, following again ideas of~\cite{dJC95}, we consider minimal sets of equations for checking $\V$-dependence for some variety~$\V$.  In particular, we provide finite minimal sets of equations for dependence in both the locally finite variety of distributive lattices and the non-locally finite variety of lattices, obtaining a first proof that the dependence problem for the variety of lattices is decidable. Finally,  in Section~\ref{sec:openproblems}, we conclude the paper with a short list of open problems.


\section{Equational consequence and free algebras}\label{sec:consequence}

Let us begin by recalling some elementary material on universal algebra, referring to~\cite{BS81} for further details and references. We assume that $\lang$ is an algebraic language and that an {\em $\lang$-algebra} $\m{A}$ is a (first-order) structure for this language with universe $A$ and fundamental operations $f^\m{A}$ for each function symbol $f$ of $\lang$. For any set of variables $\xbar$, we denote by $\Tmc(\xbar)$ the set of {\em $\lang$-terms over $\xbar$}, and by $\Eqc(\xbar) := \Tmc(\xbar)\times \Tmc(\xbar)$, the set of {\em $\lang$-equations over $\xbar$}. For $\Tmc(\xbar)\neq\emptyset$ (i.e., when $\xbar\neq\emptyset$ or $\lang$ contains a constant), we denote by $\Tm(\xbar)$ the \emph{$\lang$-term}  {\em algebra over $\xbar$}.  We also write $t(\xbar)$,  $\eps(\xbar)$, or $\Si(\xbar)$ to mean that the variables occurring in an $\lang$-term $t$, an $\lang$-equation $\eps$, or a set of $\lang$-equations $\Si$, are included in $\xbar$, and assume that $\xbar$, $\ybar$, etc. are disjoint sets, writing $\xbar,\ybar$ to denote their disjoint union. Given an $\lang$-term $t(x_1,\ldots,x_n)$ and an $\lang$-algebra $\m{A}$, we denote by $t^\m{A}$ the induced {\em term-function} from $A^n$ to $A$.

For any homomorphism $h\colon \m{A}\to\m{B}$ between $\lang$-algebras $\m{A}$ and $\m{B}$, its {\em kernel} $\ker(h) := \{\lan a,b\ran\in A\times A \mid h(a)=h(b)\}$ forms a {\em congruence} of $\m{A}$: that is, an equivalence relation that is preserved by the fundamental operations of $\m{A}$. Moreover, the converse is also true; every congruence of $\m{A}$ is the kernel of some homomorphism with domain $\m{A}$. For any $\lang$-algebra $\m{A}$, the set $\con{\m{A}}$ of congruences of $\m{A}$ forms a complete lattice ordered by inclusion $\subseteq$ with greatest element $\nabla_A := A\times A$ and least element $\Delta_A := \{\lan a,a \ran\mid a\in A\}$. Given $S \subseteq A \times A$,  the congruence  $\cg{\m{A}}(S)$ of $\m{A}$ {\em generated by $S$} is the smallest congruence of $\m{A}$ containing $S$. A congruence $\The$ of $\m{A}$ is {\em finitely generated} if $\The=\cg{\m{A}}(S)$ for some finite $S\subseteq A\times A$.

Let $\cop{H}$, $\cop{I}$, $\cop{S}$, and $\cop{P}$ denote the class operators of taking homomorphic images, isomorphic images, subalgebras, and products, respectively. A class of $\lang$-algebras $\K$ is called a {\em variety} if it is closed under $\cop{H}$, $\cop{S}$, and $\cop{P}$. By theorems of Birkhoff and Tarski, respectively, $\K$ is a variety if, and only if, it is an equational class, and $\cop{HSP}(\K)$ is the variety generated by~$\K$.

\begin{example}
A {\em Heyting algebra} is an algebra $\lan H,\mt,\jn,\to,0,1\ran$ such that $\lan H,\mt,\jn,0,1 \ran$ is a bounded distributive lattice (with $a \le b :\Leftrightarrow a\land b=a$) and $\to$ is the residual of $\mt$; that is, $a \le b\to c$ if, and only if, $a\mt b \le c$ for all $a,b,c\in H$. Heyting algebras form a variety $\cls{HA}$ that provides the algebraic semantics for intuitionistic propositional logic.
\end{example}

Equational consequence for a class $\K$ of $\lang$-algebras may be defined as follows. For a set of $\lang$-equations $\Si \cup \{\eps\}$ containing exactly the variables in a set $\xbar$, 
\[
\begin{array}{rcc}
\Si \mdl{\K} \eps & :\Longleftrightarrow & \text{for every $\m{A} \in \K$ and homomorphism $e
\colon \Tm(\xbar) \to \m{A}$,}\\[.05in] 
& &  \Si \subseteq \ker(e) \enspace \Longrightarrow \enspace \eps \in \ker(e).
\end{array}
\]
For a set of $\lang$-equations $\Si \cup \De$, we write $\Si \mdl{\K} \De$ if $\Si \mdl{\K} \eps$ for all $\eps \in \De$. 

In the case where $\K$ is a variety, we may reformulate equational consequence in terms of congruences of $\K$-free algebras. Let us first recall the construction of $\K$-free algebras for an arbitrary class $\K$ of $\lang$-algebras over a set $\xbar$, assuming that either $\lang$ contains a constant or $\xbar$ is non-empty. Let $\The_\K(\xbar)$ be the smallest congruence of $\Tm(\xbar)$ such that the quotient by this congruence embeds into a member of $\K$, i.e.,
\[
\The_\K(\xbar) := \bigcap \{\The \in\con{\Tm (\xbar)}  \mid \Tm (\xbar)/\The \in \cop{IS}(\K)\}.
\]
The {\em $\K$-free algebra over $\xbar$} may then be defined as
\[
\m{F}_{\K}(\xbar)  := \Tm(\xbar) / \The_{\K}(\xbar).
\]
It follows that for any $s,t\in \Tmc(\xbar)$,
\[
\lan s,t \ran\in\The_\K(X)\enspace \iff\enspace\: \mdl{\K} s\eq t \enspace \iff\enspace\: \mdl{\m{F}_{\K}(\xbar)}s \eq t.
\]
Where appropriate, we deliberately confuse $\lang$-terms, $\lang$-equations, and sets of $\lang$-equations with the corresponding elements, pairs of elements, and sets of pairs of elements from $\F_\K(\xbar)$. Hence for $s,t\in \Tmc(\xbar)$,
\[
\mdl{\K} s\eq t \enspace \iff\enspace s=t \: \text{ in $\m{F}_{\K}(\xbar)$.}
\]
Also, when the class of algebras $\K$ is clear from the context, we drop the subscript and write simply $\F(\xbar)$.

The following lemma expresses the crucial connection between equational consequence in a variety and congruences on the free algebras of that variety.

\begin{lemma}[c.f.~{\cite{MMT14}*{Lemma~2}}]\label{l:eqconseq}
For any variety $\V$ and $\Si\cup\De \subseteq \Eqc(\xbar)$,
\[
\Si \models_\V \De\enspace \iff\enspace \cg{\F(\xbar)}(\De) \subseteq \cg{\F(\xbar)} (\Si).
\]
\end{lemma}

In what follows, we will omit mention of the language $\lang$, assuming throughout that a class of algebras $\K$ is a class of $\lang$-algebras, and that terms, equations, and sets of equations are defined over this language. We will also adopt the useful notation $[n]$ to denote the set $\{1,\dots,n\}$ for $n\in\N$.


\section{An algebraic theory of dependence}\label{sec:algebraic}

We now have the tools available to formulate and study the de~Jongh notion of dependence in a more general algebraic setting. Let $\V$ be any variety. We call terms $t_1,\dots,t_n \in \Tmc(\xbar)$  $\V$-{\em dependent} if for some equation $\eps(y_1,\dots,y_n)$,
\[
\mdl{\V} \eps(t_1,\dots,t_n)\quad \text{and} \quad \not\mdl{\V} \eps;
\]
otherwise, we call $t_1,\dots,t_n$ $\V$-{\em independent}.

\begin{example}
If $\V$ is the variety of vector spaces over some fixed field $\m{K}$, this notion of independence coincides with the usual notion  in linear algebra. Just observe that terms $t_1,\dots,t_n \in \Tmc(\xbar)$ in the language with the usual group operations and scalar multiplication for each $\lambda\in K$ are $\V$-independent if, and only if, (without loss of generality) for any  equation $\eps(y_1,\dots,y_n)$ of the form $\lambda_1 y_1 + \dots + \lambda_n y_n \eq 0$ with $\lambda_1,\dots,\lambda_n\in K$,
\begin{align*}
\mdl{\V} \lambda_1 t_1 + \dots + \lambda_n t_n \eq 0 
& \enspace \Longrightarrow \enspace\: \mdl{\V} \lambda_1 y_1 + \dots + \lambda_n y_n \eq 0,
\end{align*}
and that $\mdl{\V} \lambda_1 y_1 + \dots + \lambda_n y_n \eq 0$ if, and only if, $\lambda_1 =  \dots =  \lambda_n = 0$.

\end{example}

\begin{example}
Let $\Lat$ be the variety of lattices and $\DLat$ the variety of distributive lattices, and consider the lattice terms
\[
t_1 := x_1 \land (x_2\lor x_3)
\quad\text{and}\quad
t_2 :=  x_2 \lor (x_1 \land x_3).
\]
Defining $\eps(y_1,y_2) := y_1 \le y_2$ (where $s\le t$ denotes $s \land t \eq s$), we have $\mdl{\DLat}\eps(t_1,t_2)$ and $\not\mdl{\DLat}\eps$, so $t_1$ and $t_2$ are $\DLat$-dependent. However, no such equation  exists in the case of lattices, so $t_1$ and $t_2$ are $\Lat$-independent. 
\end{example}

The next result provides equivalent characterizations of $\V$-independence.

\begin{proposition}\label{p:indpcechar}
Let $\V$ be a variety. For any terms $t_1,\dots,t_n \in \Tmc(\xbar)$ and variables $\ybar = \{y_1,\dots,y_n\}$, the following are equivalent:
\begin{rlist}
\item[\rm (1)]	$t_1,\dots,t_n$ are $\V$-independent.
\item[\rm (2)]	The homomorphism $h \colon \F(\ybar) \to \F(\xbar)$ defined by mapping $y_i$ to $t_i$ for each $i\in[n]$ is injective.
\item[\rm (3)]	$\cg{\F(\xbar,\ybar)}(\lan y_1, t_1\ran,\dots,\lan y_n,t_n\ran) \cap F(\ybar)^2 = \De_{F(\ybar)}$.
\end{rlist}
\end{proposition}
\begin{proof}\text{ }\\[.05in]
(1)\,$\Rightarrow$\,(2)\, Suppose that $t_1,\dots,t_n$ are $\V$-independent and that $h(u)=h(v)$ in $ \F(\xbar)$ for some $u,v\in \Tmc(\ybar)$. That is, $u(t_1,\dots,t_n) = v(t_1,\dots,t_n)$ in $\F(\xbar)$, and hence $\mdl{\V} u(t_1,\dots,t_n) \eq v(t_1,\dots,t_n)$.  But $t_1,\dots,t_n$ are $\V$-independent, so $\mdl{\V} u\eq v$ and $u=v$ in $\F(\ybar)$. Hence $h$ is injective.

(2)\,$\Rightarrow$\,(3)\,  Suppose that $h$ is injective. Let $f \colon \F(\ybar) \to \F(\xbar,\ybar)$ be the natural inclusion homomorphism and let $g\colon \F(\xbar,\ybar)\to \F(\xbar)$ be the homomorphism mapping $x\in\xbar$ to $x$ and $y_i$ to $t_i$ for each $i\in[n]$. Clearly, $h = g\circ f$ and $\cg{\F(\xbar,\ybar)}(\lan y_1, t_1\ran,\dots,\lan y_n,t_n\ran)\subseteq\ker(g)$, and hence, since $h$ is injective,
\[
\cg{\F(\xbar,\ybar)}(\lan y_1, t_1\ran,\dots,\lan y_n,t_n\ran) \cap F(\ybar)^2 \subseteq \ker(h) = \De_{F(\ybar)}.
\]
(3)\,$\Rightarrow$\,(1)\,
Suppose that $\cg{\F(\xbar,\ybar)}(\lan y_1, t_1\ran,\dots,\lan y_n,t_n\ran) \cap F(\ybar)^2 = \De_{F(\ybar)}$ and consider $\eps\in\Eqc(\ybar)$ with $\mdl{\V} \eps(t_1,\dots,t_n)$. Then $\{y_1\eq t_1,\dots,y_n\eq t_n\}\mdl{\V}\eps$ and an application of Lemma~\ref{l:eqconseq} yields
\[
\eps \in \cg{\F(\xbar,\ybar)}(\lan y_1, t_1\ran,\dots,\lan y_n,t_n\ran) \cap F(\ybar)^2 = \De_{F(\ybar)}.
\]
Hence, by Lemma~\ref{l:eqconseq}, also $\mdl{\V} \eps$. So $t_1,\dots,t_n$ are $\V$-independent.
\end{proof}

\begin{remark}\label{rem:independence}
Condition (3) of Proposition~\ref{p:indpcechar} can also be understood as a property of equational consequence. Given terms $t_1,\dots,t_n \in \Tmc(\xbar)$ and variables $\ybar = \{y_1,\dots,y_n\}$, define
\[
\Ga := \{y_1 \eq t_1,\dots,y_n \eq t_n\} \quad\text{and}\quad 
\Pi := \{\eps\in \Eqc(\ybar) \mid \Ga \mdl{\V} \eps\}.
\]
Then for any $\eps\in \Eqc(\ybar)$,
\[
\mdl{\V} \eps(t_1,\dots,t_n) \enspace \iff\enspace  \Ga\mdl{\V} \eps\enspace \iff\enspace \Pi \mdl{\V} \eps,
\]
and, corresponding to condition (3),
\[
t_1,\dots,t_n \text{ are $\V$-independent} \enspace \iff \enspace\: \mdl{\V}\Pi.
\]
\end{remark}

Let us now consider the relationship between this notion of dependence and the general algebraic notion introduced by Marczewski in~\cites{Mar58}. We say that elements $a_1,\dots,a_n$ of an algebra $\m{A}$ are {\em Marczewski-dependent} in $\m{A}$ if there exist terms $u(y_1,\dots,y_n),v(y_1,\dots,y_n)$ satisfying
\[
u^\m{A}(a_1,\dots,a_n) = v^\m{A}(a_1,\dots,a_n) \quad\text{and}\quad u^\m{A} \neq v^\m{A};
\]
otherwise, we call $a_1,\dots,a_n$  {\em Marczewski-independent}  in $\m{A}$.

\begin{remark}
Equivalently, $a_1,\dots,a_n \in A$ are  Marczewski-independent in $\m{A}$ if, and only if, they are distinct and generate a subalgebra of $\m{A}$ that is $\cop{HSP}(\m{A})$-free over the set of generators $\{a_1,\dots,a_n\}$~\cite{Mar59}.
\end{remark}

It is not hard to see that $\V$-dependence for a variety $\V$ corresponds to the Marczewski-dependence of elements of finitely generated $\V$-free algebras.

\begin{proposition}
Let $\V$ be a variety. Terms $t_1,\dots,t_n \in \Tmc(\xbar)$ are $\V$-dependent if, and only if, $t_1,\dots,t_n$ are Marczewski-dependent in $\F(\xbar)$.
\end{proposition}
\begin{proof}
It suffices to observe that $t_1,\dots,t_n\in \Tmc(\xbar)$ are Marczewski-dependent in $\F(\xbar)$ if, and only if,  there exist terms $u(y_1,\dots,y_n),v(y_1,\dots,y_n)$ satisfying
\[
u^{\F(\xbar)}(t_1,\dots,t_n) = v^{\F(\xbar)}(t_1,\dots,t_n) \quad\text{and}\quad u^{\F(\xbar)} \neq v^{\F(\xbar)},
\]
or, equivalently, there exist terms $u(y_1,\dots,y_n),v(y_1,\dots,y_n)$ satisfying
\[
\mdl{\V} u(t_1,\dots,t_n) \eq v(t_1,\dots,t_n) \quad\text{and}\quad \mdl{\V} u\not\eq v,
\]
which holds if, and only if, $t_1,\dots,t_n$ are $\V$-dependent.
\end{proof}

Let us now consider a more general version of dependence defined for a variety $\V$ relative to a fixed set of equations. We call $t_1,\dots,t_n \in \Tmc(\xbar)$  $\V$-{\em dependent over $\Si\subseteq\Eqc(\xbar)$} if for some equation $\eps(y_1,\dots,y_n)$,
\[
\Si \mdl{\V} \eps(t_1,\dots,t_n)\quad \text{and} \quad \not\mdl{\V} \eps;
\]
otherwise, we call $t_1,\dots,t_n$ $\V$-{\em independent over $\Si$}.

To obtain a reformulation of this property analogous to Proposition~\ref{p:indpcechar}, we consider for $\Si\subseteq\Eqc(\xbar)$, the quotient algebra $\F(\xbar)/\cg{\F(\xbar)}(\Si)$, denoting its elements by $[t]_\Si$ for each $t\in \Tmc(\xbar)$.

\begin{proposition}\label{p:sigmaindpcechar}
Let $\V$ be a variety. For any $\Si\subseteq\Eqc(\xbar)$, $t_1,\dots,t_n \in \Tmc(\xbar)$, and $\ybar = \{y_1,\dots,y_n\}$, the following are equivalent:
\begin{rlist}
\item[\rm (1)]	$t_1,\dots,t_n$ are $\V$-independent over $\Si$.
\item[\rm (2)]	The homomorphism $h \colon \F(\ybar)\to\F(\xbar)/\cg{\F(\xbar)}(\Si)$ mapping $y_i$ to $[t_i]_\Si$ for each $i\in[n]$ is injective.
\item[\rm (3)]	$\cg{\F(\xbar,\ybar)}(\Si\cup\{\lan y_1, t_1\ran,\dots,\lan y_n,t_n\ran\}) \cap F(\ybar)^2 = \De_{F(\ybar)}$.
\end{rlist}
\end{proposition}
\begin{proof}\text{ }\\[.05in]
(1)\,$\Rightarrow$\,(2)\, Suppose that $t_1,\dots,t_n$ are $\V$-independent over $\Si$ and  $h(u)=h(v)$ for some $u,v\in \Tmc(\ybar)$. It follows that $[u(t_1,\dots,t_n)]_\Si = [v(t_1,\dots,t_n)]_\Si$ in $\F(\xbar)/\cg{\F(\xbar)}(\Si)$ and, by Lemma~\ref{l:eqconseq}, that $\Si\mdl{\V} u(t_1,\dots,t_n) \eq v(t_1,\dots,t_n)$. Since, by assumption, $t_1,\dots,t_n$ are $\V$-independent over $\Si$, we have $\mdl{\V} u\eq v$. So $u=v$ in $\F(\ybar)$. That is, $h$ is injective.

(2)\,$\Rightarrow$\,(3)\, Suppose that $h$ is injective. Let $f \colon \F(\ybar) \to \F(\xbar,\ybar)$ be the natural inclusion homomorphism and let $g\colon \F(\xbar,\ybar)\to\F(\xbar)/\cg{\F(\xbar)}(\Si)$ be the homomorphism mapping $x\in\xbar$ to $[x]_\Si$ and $y_i$ to $[t_i]_\Si$ for each $i\in[n]$. Clearly, $h = g\circ f$ and $\cg{\F(\xbar,\ybar)}(\Si\cup\{\lan y_1, t_1\ran,\dots,\lan y_n,t_n\ran\})\subseteq\ker(g)$, and hence, since $h$ is injective,
\[
\cg{\F(\xbar,\ybar)}(\Si\cup\{\lan y_1, t_1\ran,\dots,\lan y_n,t_n\ran\}) \cap F(\ybar)^2 \subseteq\ker(h) = \De_{F(\ybar)}.
\]
(3)\,$\Rightarrow$\,(1)\,
Suppose that $\cg{\F(\xbar,\ybar)}(\Si\cup\{\lan y_1, t_1\ran,\dots,\lan y_n,t_n\ran\}) \cap F(\ybar)^2 = \De_{F(\ybar)}$ and $\Si\mdl{\V} \eps(t_1,\dots,t_n)$ for some $\eps\in\Eqc(\ybar)$. Then $\Si\cup\{y_1\eq t_1,\dots,y_n\eq t_n\}\mdl{\V}\eps$ and, by Lemma~\ref{l:eqconseq},
\[
\eps \in \cg{\F(\xbar,\ybar)}(\Si\cup\{\lan y_1, t_1\ran,\dots,\lan y_n,t_n\ran\}) \cap F(\ybar)^2=\De_{F(\ybar)}.
\]
Hence, by Lemma~\ref{l:eqconseq}, also $\mdl{\V} \eps$. So $t_1,\dots,t_n$ are $\V$-independent over $\Si$.
\end{proof}

\begin{remark}\label{rem:independenceoversigma}
$\V$-dependence over $\Si$ can again be understood as a property of equational consequence. For $t_1,\dots,t_n \in \Tmc(\xbar)$ and  $\ybar = \{y_1,\dots,y_n\}$,  let
\[
\Ga := \Si \cup \{y_1 \eq t_1,\dots,y_n \eq t_n\} \quad\text{and}\quad \Pi := \{\eps\in \Eqc(\ybar) \mid \Ga\mdl{\V} \eps\}.
\]
Then for any $\eps\in \Eqc(\ybar)$,
\[
\Si \mdl{\V} \eps(t_1,\dots,t_n) \enspace \iff\enspace  \Ga\mdl{\V} \eps\enspace \iff\enspace \Pi \mdl{\V} \eps,
\]
and, corresponding to condition (3),
\[
t_1,\dots,t_n \text{ are $\V$-independent over $\Si$} \enspace \iff \enspace\: \mdl{\V}\Pi.
\]
\end{remark}

A natural question to ask at this point is whether this more general notion is related to Marczewski-dependence. Below we show that this is indeed the case, although the relationship is unlikely to be of any practical value. 

Recall that the {\em positive diagram} $\DiagpA$ of an algebra $\m{A}$ can be identified with the set of equations $s(a_1,\dots,a_n) \eq t(a_1,\dots,a_n) \in \Eqc(A)$ such that $s^\m{A}(a_1,\dots,a_n) = t^\m{A}(a_1,\dots,a_n)$.

\begin{proposition}
Let $\m{A}$ be any algebra. Then $a_1,\dots,a_n\in A$ are Marczewski-dependent in $\m{A}$ if, and only if, $a_1,\dots,a_n\in\Tmc(A)$ are $\cop{HSP}(\m{A})$-dependent over  $\DiagpA$. 
\end{proposition}
\begin{proof}
It suffices to observe that for any set of variables $\ybar = \{y_1,\dots,y_n\}$, equation $u \eq v \in \Eqc(\ybar)$, and  $a_1,\dots,a_n\in A$,
\begin{rlist}
\item[\rm (i)]	$\DiagpA \mdl{\cop{HSP}(\m{A})} u(a_1,\dots,a_n) \eq v(a_1,\dots,a_n)$
			\begin{align*}
			\iff & u(a_1,\dots,a_n) \eq v(a_1,\dots,a_n) \in \DiagpA\\
			\iff & u^\m{A}(a_1,\dots,a_n)=v^\m{A}(a_1,\dots,a_n)
			\end{align*}
\item[\rm (ii)]	$\mdl{\cop{HSP}(\m{A})} u\eq v \iff \m{A}\models u \eq v \iff u^\m{A}=v^\m{A}$. \qedhere
\end{rlist}
\end{proof}


\section{Deciding dependence}\label{sec:deciding}

In~\cite{dJC95} de~Jongh and Chagrova established the decidability of $\lgc{IPC}$-dependence for finitely many formulas (equivalently, the $\cls{HA}$-dependence of finitely many terms) using Pitts' constructive proof of uniform interpolation for $\lgc{IPC}$~\cite{Pit92}. More concretely, they proved that formulas $\f_1,\dots,\f_n$ are $\lgc{IPC}$-independent if, and only if, the right uniform interpolant of $(p_1\lra\f_1) \land \dots \land (p_n\lra\f_n)$ with respect to a new set of variables $\{p_1,\ldots,p_n\}$ is a theorem of $\lgc{IPC}$. In this section, we generalize their approach to an arbitrary variety $\V$ and finite set of equations $\Si$, showing that to check the $\V$-dependence of terms $t_1,\dots,t_n$ over $\Si$, a constructive proof of the weaker property of {\em coherence} in $\V$ --- or, when $\Si=\emptyset$, just coherence for finitely generated $\V$-free algebras --- suffices. 

A finitely presented algebra $\m{A}\in\V$ is said to be {\em coherent} if every finitely generated subalgebra of $\m{A}$ is finitely presented, and a variety $\V$ is called coherent if all of its finitely presented members are coherent. The following result from~\cite{KM18} relates this notion to finitely generated congruences on finitely generated $\V$-free algebras and equational consequence.

\begin{theorem}[{\cite{KM18}*{Theorem~2.3}}]\label{thm:coherence}
Let $\V$ be a variety. The following are equivalent:
\begin{rlist}
\item[\rm (1)]	$\V$ is coherent.
\item[\rm (2)]	For any finite sets $\xbar,\ybar$ and finitely generated congruence $\The$ of $\F(\xbar,\ybar)$, the congruence $\The\cap F(\ybar)^2$ on $\F(\ybar)$ is finitely generated.
\item[\rm (3)]	For any finite sets $\xbar,\ybar$ and finite set of equations $\Ga(\xbar,\ybar)$, there exists a finite set of equations $\Pi(\ybar)$ such that for any equation $\eps(\ybar)$,
\[
\Ga\mdl{\V}\eps\enspace\iff\enspace\Pi\mdl{\V}\eps.
\]
\end{rlist}
\end{theorem}

\begin{remark}
Condition (3) is closely related to the property of {\em right uniform deductive interpolation}, which is obtained by replacing ``any equation $\eps(\ybar)$'' with ``any equation $\eps(\ybar,\zbar)$''. Indeed, a variety has right uniform deductive interpolation if, and only if, it is coherent and has deductive interpolation. Let us note also that deductive interpolation (obtained from right uniform deductive interpolation by dropping the requirement that $\Pi(\ybar)$ be finite) is equivalent to the amalgamation property in the presence of the congruence extension property. We refer to~\cites{MMT14,vGMT17,KM18} for further details and discussion of these relationships.
\end{remark}

Let us call the problem of finding for any finite sets $\xbar,\ybar$ and finite set of equations $\Ga(\xbar,\ybar)$, a finite set of equations $\Pi(\ybar)$ satisfying the equivalence in condition (3) of Theorem~\ref{thm:coherence}, the {\em coherence problem} for $\V$. Note that this is equivalent to the problem of finding a finite presentation for a finitely generated subalgebra of a finitely presented algebra of $\V$.

\begin{proposition}\label{prop:decidable}
Let $\V$ be a variety and let $\Si$ be any finite set of equations. If the coherence problem for $\V$ and the equational theory of $\V$ are both decidable, then $\V$-dependence over $\Si$ is decidable.
\end{proposition}
\begin{proof}
Given $t_1,\dots,t_n \in \Tmc(\xbar)$, define $\Ga := \Si\cup\{y_1\eq t_1,\dots,y_n\eq t_n\}$. It follows from Remark~\ref{rem:independenceoversigma} that $t_1,\dots,t_n$ are $\V$-independent over $\Si$ if, and only if,  for any equation $\eps(\ybar)$,
\[
\Ga\mdl{\V}\eps\enspace\iff\enspace\:\mdl{\V}\eps.
\]
However, by assumption, a finite set of equations $\Pi(\ybar)$ can be constructed such that for any equation $\eps(\ybar)$,
\[
\Ga\mdl{\V}\eps\enspace\iff\enspace\Pi\mdl{\V}\eps.
\]
It follows that $t_1,\dots,t_n\in \Tmc(\xbar)$ are $\V$-independent over $\Si$ if, and only if, $\mdl{\V}\Pi$, which, by assumption, is decidable. 
\end{proof}

The coherence problem is clearly decidable for any locally finite variety. Also, as mentioned above, Pitts' constructive proof of uniform  interpolation for $\lgc{IPC}$ provides an algorithm that decides the coherence problem for the variety $\cls{HA}$ of Heyting algebras~\cite{Pit92} and hence also an algorithm for deciding  $\cls{HA}$-dependence over any finite set of equations. 

Pitts-style proofs of uniform interpolation have been obtained for various intermediate, modal, and substructural logics (see in particular~\cites{Vis96,Sha93,Bil07,ADO14}), but typically establish an implication-based uniform interpolation property that does not imply coherence. However, for the modal logics $\lgc{GL}$ and $\lgc{S4Grz}$, the proof-theoretic proofs of uniform interpolation by Bilkov{\'a}~\cite{Bil07} (originally proved semantically by Shavrukov~\cite{Sha93} and Visser~\cite{Vis96}, respectively) provide constructive proofs of coherence for the associated varieties and hence also decidability of dependence over finite sets of equations.

\begin{example}[Abelian $\ell$-groups]
An {\em abelian $\ell$-group} is an algebraic structure $\lan L,\mt,\jn,+,-,0 \ran$ such that $\lan L,\mt,\jn\ran$ is a lattice with order $a\le b:\Leftrightarrow a\mt b =a$, $\lan L,+,-,0 \ran$ is an abelian group, and $a\leq b$ implies $a+c\leq b+c$ for all $a,b,c\in L$. They form a variety $\cls{LA}$ that is generated as a quasivariety by $\m{R}=\langle \mathbb{R},\land,\lor,+,-,0 \rangle$ (cf.~\cite{AF88}*{Lemma~6.2}) and has a decidable equational theory. In~\cite{Mun17} it is proved that checking whether the subalgebra generated by $n$ elements of a finitely generated free abelian $\ell$-group is isomorphic to the $n$-generated free abelian $\ell$-group is decidable, and hence, although this is not explicitly stated, that $\cls{LA}$-dependence is decidable. However, a stronger version of this result, the decidability of $\cls{LA}$-dependence over a finite set of equations, follows already from a (quite easy) constructive proof of coherence given implicitly in~\cite{MMT14}. Note first that it suffices to show that for any finite set $\ybar$ and finite set of equations $\Ga(x,\ybar)$, there exists a finite set of equations $\Pi(\ybar)$ such that for any equation $\eps(\ybar)$,
\[
\Ga\mdl{\m{R}}\eps\iff\Pi\mdl{\m{R}}\eps.
\]
Moreover, we may assume (with a little work, omitted here) that $\Ga$ consists of inequations $0 \le s_i + nx$ ($i \in I$), $0 \le t_j - nx$ ($j \in J$), and $0 \le u_k$ ($k \in K$) for some $n \ge 1$, finite sets $I,J,K$, and terms $s_i(\ybar),t_j(\ybar),u_k(\ybar)$. The desired set $\Pi(\ybar)$ is then $\{0 \le s_i + t_j \mid i \in I; j \in J\}\cup\{0 \le u_k\mid k \in K\}$. 
\end{example}

\begin{example}[MV-algebras]
The variety $\cls{MV}$ of {\em MV-algebras} consists of algebraic structures $\lan M,\oplus,\lnot,0\ran$ satisfying the equations
\[
\begin{array}{rlrl}
{\rm (M1)} &  x\oplus(y\oplus z)\eq(x\oplus y)\oplus z &  {\rm (M4)} &  \lnot\lnot x\eq x\\
{\rm (M2)} &  x\oplus y\eq y\oplus x  & {\rm (M5)}  &  x\oplus\lnot0 \eq\lnot 0\\
{\rm (M3)} &  x\oplus0\eq x & {\rm (M6)} &  \lnot(\lnot x\oplus y)\oplus y\eq\lnot(\lnot y\oplus x)\oplus x.
\end{array}
\]
It is generated as a quasivariety by $\m{[0,1]} = \langle [0,1],\oplus,\lnot,0 \rangle$, where $a \oplus b = \min(1,a+b)$ and $\lnot a = 1-a$, with defined operations $1 := \lnot 0$, $a \odot b := \lnot (\lnot a \oplus \lnot b)$, $a \lor b := \lnot (\lnot a \oplus b) \oplus b$, and $a \land b := \lnot (\lnot a \lor \lnot b)$~\cite{diN91}. Let $\m{R}^u$ be the unital abelian $\ell$-group consisting of $\m{R}$ with an additional constant $1$.  It follows from McNaughton's representation theorem (or see~\cite{MMT14}*{Sec.~6} for a direct proof) that (i) the interpretation of any term $s$ in $\m{[0,1]}$ is equivalent on $[0,1]$ to the interpretation of some term $(t \land 0) \lor 1$ in $\m{R}^u$, and, conversely, (ii) the interpretation of any term $(t \land 0) \lor 1$ in $\m{R}^u$ is equivalent on $[0,1]$ to the interpretation of some term $s$ in $\m{[0,1]}$. Coherence for $\cls{MV}$ may then be established constructively as in the case of abelian $\ell$-groups described in the previous example. Since $\cls{MV}$ also has a decidable equational theory, $\cls{MV}$-dependence over a finite set of equations is decidable.  
\end{example}

As shown in~\cites{KM18,KM18a}, coherence for a non-locally finite variety is a rather exceptional property. In particular, the varieties of lattices, semigroups, and groups, as well as broad families of varieties of modal algebras and residuated lattices are not coherent. However, for checking $\V$-dependence (i.e., over the empty set of equations), it is not necessary to have an algorithm that decides the full coherence problem; it suffices to consider the coherence of finitely generated $\V$-free algebras.

\begin{lemma}
Let $\V$ be a variety. The following are equivalent for any finite set $\xbar$:
\begin{rlist}
\item[\rm (1)]	$\F(\xbar)$ is coherent.
\item[\rm (2)]	For any\, $t_1,\dots,t_n \in \Tmc(\xbar)$ and\, $\ybar = \{y_1,\dots,y_n\}$, the congruence $\cg{\F(\xbar,\ybar)}(\lan y_1, t_1\ran,\dots,\lan y_n,t_n\ran) \cap F(\ybar)^2$ is finitely generated.
\item[\rm (3)]	For any\, $t_1,\dots,t_n \in \Tmc(\xbar)$ and\, $\ybar = \{y_1,\dots,y_n\}$, there exists a finite set of equations $\Pi(\ybar)$ such that for any equation $\eps(\ybar)$,
\[
\{y_1\eq t_1,\dots,y_n\eq t_n\}\mdl{\V}\eps\enspace\iff\enspace\Pi\mdl{\V}\eps.
\]
\end{rlist}
\end{lemma}
\begin{proof} 
For the equivalence of (1) and (2), consider any $t_1,\dots,t_n\in \Tmc(\xbar)$ and\, $\ybar = \{y_1,\dots,y_n\}$. Let  $h \colon \F(\ybar) \to \F(\xbar)$ be the homomorphism mapping $y_i$ to $t_i$ for each $i\in[n]$. Then $\ker(h) = \cg{\F(\xbar,\ybar)}(\lan y_1, t_1\ran,\dots,\lan y_n,t_n\ran) \cap F(\ybar)^2$ and, by the homomorphism theorem  (see~\cite{BS81}), the quotient $\F(\ybar) / \ker(h)$ is isomorphic to the subalgebra $\m{A}$ of $\F(\xbar)$ generated by $\{t_1,\dots,t_n\}$. We now recall the following:
\begin{fact}[\cite{KM18}*{Lemma~2.2}]
If $\m{B}\in\V$ is finitely presented and isomorphic to a quotient $\F(\zbar)/\The$ for some finite set $\zbar$, then $\The$ is finitely generated.
\end{fact}
Hence, if $\F(\xbar)$ is coherent, $\m{A}$ is finitely presented and, by the fact stated above, $\cg{\F(\xbar,\ybar)}(\lan y_1, t_1\ran,\dots,\lan y_n,t_n\ran) \cap F(\ybar)^2$ is finitely generated. Conversely, assuming (2), if $\m{A}$ is a subalgebra of $\F(\xbar)$ generated by $\{t_1,\dots,t_n\}$, then the congruence $\cg{\F(\xbar,\ybar)}(\lan y_1, t_1\ran,\dots,\lan y_n,t_n\ran) \cap F(\ybar)^2$ is finitely generated and $\m{A}$ is finitely presented, so $\F(\xbar)$ is coherent. The equivalence of (2) and (3) follows directly from Lemma~\ref{l:eqconseq}.
\end{proof}

Let us therefore call the problem of finding for any $t_1,\dots,t_n \in \Tmc(\xbar)$ and $\ybar = \{y_1,\dots,y_n\}$, a finite set of equations $\Pi(\ybar)$ satisfying the equivalence in condition (3), the {\em free coherence problem} for $\V$. (Note that this is equivalent to finding a finite presentation for a finitely generated subalgebra of a finitely generated $\V$-free algebra.) It follows directly from the proof of Proposition~\ref{prop:decidable} that if the free coherence problem for $\V$ and the equational theory of $\V$ are decidable, then $\V$-dependence is decidable.

\begin{example}[Groups]
The variety $\cls{G}\mathit{rp}$ of {\em groups} is not coherent; e.g., the wreath product $\mathbb{Z}\,{\rm wr}\,\mathbb{Z}$ is a finitely generated subgroup of a finitely presented group that is not finitely presented~\cite{Cle06}. However, by the Nielsen-Schreier theorem, every finitely generated subgroup of a finitely generated free group is again a finitely generated free group, so finitely generated free groups are coherent. Moreover, Nielsen's proof in~\cite{Nie21} also determines the rank of a finitely generated subgroup of a free group, so the free coherence problem for $\cls{G}\mathit{rp}$ is decidable. Hence also $\cls{G}\mathit{rp}$-dependence is decidable.
\end{example}

It may not be the case that all finitely generated free algebras of a variety $\V$ are coherent. However, assuming that the equational theory of $\V$ is decidable, $\V$-dependence is decidable (again considering the proof of Proposition~\ref{prop:decidable}) if there exists an algorithm to check for $t_1,\dots,t_n \in \Tmc(\xbar)$ and $\ybar = \{y_1,\dots,y_n\}$ whether for any equation $\eps(\ybar)$,
\[
\{y_1\eq t_1,\dots,y_n\eq t_n\}\mdl{\V}\eps\enspace\iff\enspace\mdl{\V}\eps.
\]

\begin{example}[Semigroups]
The variety $\cls{SG}$ of {\em semigroups} is not coherent. Indeed, even the $n$-generated free semigroups are not coherent for $n\ge 3$; e.g., the subsemigroup of the free semigroup $\F(x,y,z)$ generated by $yx$, $yx^2$, $x^3$, $xz$, and $x^2z$ is not finitely presented~\cite{CFR96}. However, the $\cls{SG}$-dependence problem corresponds precisely to the problem of checking whether a finite subset $X$ of a finitely generated free semigroup $\F(\xbar)$ is a {\em code}, solved by the famous Sardinas-Patterson algorithm~\cite{SP53}. For finite subsets $Y,Z\subseteq F(\xbar)$, let $Y^{-1}Z := \{t \in F(\xbar) \mid st \in Z;\, s\in Y\}$. The algorithm starts with $U_0 := X$ and continues iteratively with $U_{n+1}:= X^{-1}U_n \cup U^{-1}_{n}X$ for each $n\in\mathbb{N}$. It can be proved that $X$ is {\em not} a code if, and only if, $X\cap U_n \neq\emptyset$ for some $n \ge 1$, and, since there can only be finitely many different $U_n$'s, the algorithm is terminating.
\end{example}

\begin{example}[Modal algebras]
{\em Modal algebras} --- Boolean algebras with an additional unary operation $\Box$ satisfying $\Box (x\land y) \eq \Box x \land \Box y$ and $\Box 1 =~1$ --- form a variety $\cls{MA}$ that provides algebraic semantics for the modal logic~$\lgc{K}$. This variety is not coherent~\cite{KM18}, and the coherence of finitely generated free modal algebras seems to be an open problem. Nevertheless, $\cls{MA}$-dependence can be decided using a bisimulation-based method given in~\cite{LW11} for calculating existing right uniform deductive interpolants for the description logic $\lgc{ALC}$. It is well known that $\lgc{ALC}$ restricted to a single role may be viewed as a syntactic variant of the logic $\lgc{K}$. Hence to decide $\cls{MA}$-dependence, it suffices to observe that modal formulas $\f_1,\dots,\f_n$ (corresponding to terms) are $\cls{MA}$-independent if, and only if, the formula $(p_1\lra\f_1) \land \dots \land (p_n\lra\f_n)$ has a right uniform deductive interpolant with respect to the new variables $p_1,\dots,p_n$ that is a theorem of $\lgc{K}$. This latter claim can be checked using the algorithm described in~\cite{LW11}.
\end{example}


\section{Dependence and minimal provability}\label{sec:minimal}

In~\cite{deJ82} it is shown that a single formula $\f$ is $\lgc{IPC}$-dependent if, and only if, either $\der{\lgc{IPC}}\lnot\lnot\f\to\f$ or $\der{\lgc{IPC}}\lnot\lnot\f$, and in~\cite{dJC95} a family $(\ps_i(p_1,p_2))_{i\in\N}$ of formulas is given such that two formulas $\f_1,\f_2$ are $\lgc{IPC}$-dependent if, and only if, $\der{\lgc{IPC}} \ps_i(\f_1,\f_2)$ for some $i\in\N$. In both cases, no proper subset of the given set of formulas suffices for checking $\lgc{IPC}$-dependence. In this section, we provide a general framework for describing such ``minimal provability'' results, obtaining a complete description and decidability of dependence for the variety of lattices.

Let $\V$ be a variety. Given any sets of equations $\Ga,\De$ with variables in a set $\ybar$, we write $\Ga\adm{\V}\De$ to denote that for any substitution (i.e., homomorphism) $\si\colon\Tm(\ybar)\to\Tm(\omega)$ extended to $\Eqc(\ybar)$ by $\si(s\eq t) = \si(s)\eq\si(t)$,
\[
\mdl{\V}\si[\Ga] \enspace\Longrightarrow \quad \mdl{\V}\si(\de) \,\text{ for some }\de\in\De.
\]
It is not hard to check (see, e.g.,~\cite{Iem16}) that $\adm{\V}$ is a (finitary) multiple-conclusion consequence relation over $\Eqc(\ybar)$; that is, for any equation $\eps$ and finite sets of equations $\Ga,\Ga',\De,\De'$ with variables in $\ybar$,
\begin{rlist}
\item[{\rm (i)}] 	$\{\eps\}\adm{\V}\{\eps\}$
\item[{\rm (ii)}] 	if $\Ga\adm{\V}\De$, then  $\Ga\cup\Ga'\adm{\V}\De\cup\De'$
\item[{\rm (iii)}] 	if $\Ga\cup\{\eps\}\adm{\V}\De$ and $\Ga'\adm{\V}\{\eps\}\cup\De'$, then $\Ga\cup\Ga'\adm{\V}\De\cup\De'$
\item[{\rm (iv)}] 	if $\Ga\adm{\V}\De$, then $\si[\Ga]\adm{\V}\si[\De]$ for any substitution $\si\colon\Tm(\ybar)\to\Tm(\ybar)$.
\end{rlist}
It is also easy to see that for any set of equations $\Ga\cup\{\eps\}$  with variables in $\ybar$,
\[
\Ga\mdl{\V}\eps \enspace\Longrightarrow \enspace\Ga\adm{\V}\{\eps\}.
\]

\begin{remark}
The relation $\adm{\V}$ describes the {\em admissibility} of universal formulas (or multiple-conclusion rules) in the variety $\V$. Indeed, for finite sets of equations $\Ga,\De$, it is the case that $\Ga\adm{\V}\De$  is equivalent to the validity of the implication from the conjunction of the equations in $\Ga$ to the disjunction of the equations in $\De$ in the free algebra $\F(\omega)$  (see, e.g.,~\cite{CM15}).
\end{remark} 

Let us call $\De\subseteq\Eqc(\ybar)$ {\em $\V$-refuting} for a set $\ybar$ if for any equation $\eps(\ybar)$,
\[
\notmdl{\V}\eps \:\iff\: \{\eps\}\adm{\V}\De,
\]
and {\em minimal} if, additionally, no proper subset of $\De$ is $\V$-refuting for $\ybar$.

\begin{lemma}
Let $\V$ be a variety and let $\De(\ybar)$ be a $\V$-refuting set of equations for $\ybar = \{y_1,\dots,y_n\}$. Then $t_1,\dots,t_n \in \Tmc(\xbar)$ are $\V$-dependent if, and only if, $\mdl{\V} \de(t_1,\dots,t_n)$ for some $\de\in\De$.
\end{lemma}
\begin{proof} 
Suppose first that $t_1,\dots,t_n$ are $\V$-dependent. Then $\mdl{\V} \eps(t_1,\dots,t_n)$ and $\not\mdl{\V}\eps$ for some equation $\eps(\ybar)$. Since $\De$ is a $\V$-refuting set of equations for $\ybar$, also $\{\eps\}\adm{\V}\De$. Hence $\mdl{\V} \de(t_1,\dots,t_n)$ for some $\de\in\De$. For the converse, suppose that $\mdl{\V} \de(t_1,\dots,t_n)$ for some $\de\in\De$. Clearly $\{\de\}\adm{\V}\De$ and hence $\not\mdl{\V}\de$, so $t_1,\dots,t_n$ are $\V$-dependent.
\end{proof}

If $\V$ has a decidable equational theory and a finite $\V$-refuting set of equations can be found for $\ybar = \{y_1,\dots,y_n\}$ for each $n\in\N$, then $\V$-dependence is clearly decidable. In the case where $\V$ is locally finite, a finite $\V$-refuting set of equations $\De_n$ can be obtained  for each $n\in\N$ and $\ybar=\{y_1,\dots,y_n\}$ by considering all pairs of distinct elements $s,t$ from the finite free algebra $\F(\ybar)$. A minimal $\V$-refuting set for $\ybar$ can then be obtained by iteratively removing any $\de\in\De_n$ such that $\{\de\}\adm{\V}\De_n\mathop{\setminus}\{\de\}$.

\begin{example}\label{exam:dlat}
Let us illustrate this idea with the simple case of the  (locally finite) variety $\DLat$ of distributive lattices. It is straightforward to show that for each $n\in\N$, 
\[
\De_n:=\big\{\bigwedge_{i \in I} y_i\, \le\!\!\!\bigvee_{j \in [n]{\setminus}I} \!\!\!\!y_j \mid \emptyset \neq I \subsetneq [n]\big\}
\]
 is a minimal $\DLat$-refuting set of equations for $\ybar = \{y_1,\dots,y_n\}$. We first observe that, using distributivity, the set of equations of the form $s\le t$, where $s$ is a join of meets of variables, and $t$ is a meet of joins of variables, is $\DLat$-refuting for $\ybar$. We then obtain the minimal $\DLat$-refuting set $\De_n$ using the fact that for $i\in\{1,2\}$,
 \[
 \{s_1 \lor s_2 \le t\} \adm{\DLat} \{s_i \le t\} \quad\text{and}\quad \{s \le t_1 \land t_2\} \adm{\DLat} \{s\le t_i\}.
 \]
It follows that $t_1,\dots,t_n \in \Tmc(\xbar)$ are $\DLat$-dependent if, and only if, for some $\emptyset \neq I \subsetneq [n]$,
\[
\mdl{\DLat} \bigwedge_{i \in I} t_i \le\!\!\!\bigvee_{j \in [n]{\setminus}I} \!\!\!\!t_j.
\]
We therefore obtain further confirmation that dependence in the variety of distributive lattices is decidable.
\end{example}

We devote the rest of this section to the more interesting case of the (non-locally finite) variety $\Lat$ of all lattices.\footnote{Some general properties of Marczewski-dependence in lattices are explored in~\cites{Sza63,Mar63}, but decidability issues are not considered in these papers.} The equational theory of $\Lat$ is decidable. However, it is not coherent. For example, consider the congruence $\The$ of $\F(x,y,z,u,w)$ generated by $\{y\le x, x\le z, x\le u\lor(w\land(u\lor x))\}$. It can be shown that the congruence $\Psi:=\The\cap F(y,z,u,w)^2$ of $\F(y,z,u,w)$ is not finitely generated, yielding a finitely generated sublattice of $\F(x,y,z,u,w)/\The$, isomorphic to  $\F(y,z,u,w)/\Psi$, that is not finitely presented~\cite{KM18}.\footnote{An earlier example given by R.~McKenzie of a finitely generated sublattice of a finitely presented lattice that is not finitely presented can be found in~\cite{FN16}.} On the other hand, since finitely generated sublattices of free lattices are projective~\cite{Kos72} and finitely generated projective algebras are finitely presented (see~\cite{FN16}), every finitely generated free lattice is coherent. To the best of our knowledge, however, there is no known algorithm that constructs a finite presentation for a finitely generated sublattice of a finitely generated free lattice.

We prove that $\Lat$-dependence is decidable by providing a finite minimal $\Lat$-refuting set of equations for $\ybar = \{y_1,\dots,y_n\}$ for each $n\in\N$. The following admissibility properties for $\Lat$ will be useful:

\begin{rlist}
\item[\rm (i)]	$\{x_1\le y, x_2\le y\}\adm{\Lat} \{x_1\lor x_2 \le y\}$
\item[\rm (ii)]	$\{x_1\lor x_2 \le y\}\adm{\Lat} \{x_1\le y\}$ \,and\, $\{x_1\lor x_2 \le y\}\adm{\Lat} \{x_2\le y\}$
\item[\rm (iii)]	$\{x \le y_1\}\adm{\Lat} \{x \le y_1\lor y_2\}$ \,and\, $\{x \le y_2\}\adm{\Lat} \{x \le y_1\lor y_2\}$
\item[\rm (iv)]	$\{x\le y_1, x\le y_2\}\adm{\Lat} \{x \le y_1\land y_2\}$
\item[\rm (v)]	$\{x \le y_1\land y_2\}\adm{\Lat} \{x\le y_1\}$ \,and\, $\{x \le y_1\land y_2\}\adm{\Lat} \{x\le y_2\}$
\item[\rm (vi)]	$\{x_1 \le y\}\adm{\Lat} \{x_1\land x_2 \le y\}$ \,and\, $\{x_2 \le y\}\adm{\Lat} \{x_1\land x_2 \le y\}$.
\end{rlist}
We will also need the fact that each generator $y$ of a free lattice is both join- and meet-irreducible, which can be expressed as follows:
\begin{rlist}
\item[\rm (vii)]	$\mdl{\V}y \le t_1\lor t_2 \enspace\Longrightarrow\:\enspace\mdl{\V}y \le t_1\, \text{ or } \mdl{\V}y \le t_2$
\item[\rm (viii)]	$\mdl{\V}s_1\land s_2 \le y \enspace\Longrightarrow\:\enspace\mdl{\V}s_1 \le y\, \text{ or } \mdl{\V}s_2 \le y$.
\end{rlist}
Finally, we will make crucial use of the following property of admissibility in lattices, known as {\em Whitman's condition}~\cite{Whi41}:
\[
\{x_1 \land x_2 \le y_1 \lor y_2\} \adm{\Lat} \{x_1 \le y_1 \lor y_2,\, x_2 \le y_1 \lor y_2,\, x_1 \land x_2 \le y_1,\,x_1 \land x_2 \le y_2\}.
\]
All these properties can be established directly as properties of free lattices (see, e.g.,~\cite{FN16}) or follow as easy consequences of the completeness of a simple analytic Gentzen system for lattices (see,~e.g.,~\cite{Met11}).

\begin{theorem}\label{thm:lat}
For each $n\in\N$, the following is a minimal $\Lat$-refuting set of equations for $\ybar = \{y_1,\dots,y_n\}$:
\[
\De_n:=\big\{y_i\le\!\!\!\!\!\bigvee_{j \in [n]\setminus\{i\}}\!\!\!\!\!\! y_j \mid i\in[n]\big\}
\:\cup\:
\big\{\!\!\!\!\!\bigwedge_{j \in [n]\setminus\{i\}}\!\!\!\!\!\! y_j \le y_i\mid i\in[n]\big\}.
\]
\end{theorem}
\begin{proof}
Note first that clearly $\not\mdl{\Lat}\de$ for each $\de\in\De_n$. Also $\{\de\}\not\adm{\V}\De_n\mathop{\setminus}\{\de\}$  for each $\de\in\De_n$. E.g., if $\de$ is $y_1 \le y_2 \lor\cdots\lor y_n$, let $\si$ be the substitution mapping $y_1$ to $(y_2 \land z) \lor\cdots\lor (y_n \land z)$ and $y_i$ to $y_i$ for $i\in\{2,\dots,n\}$. Then $\mdl{\Lat}\si(\de)$, but $\not\mdl{\Lat}\si(\eps)$ for each $\eps\in\De_n$.

Hence it suffices now without loss of generality (since an equation $s \eq t$ can always be replaced by inequations $s\le t$ and $t\le s$) to prove that for any inequation $\eps(\ybar)$,
\[
\notmdl{\V}\eps\enspace \Longrightarrow\enspace \{\eps\}\adm{\V}\De_n,
\]
proceeding by induction on the number of symbols in $\eps$. For the base case, observe that if $\eps$ is an inequation with one variable on the left and a join of variables on the right, or one variable on the right and a meet of variables on the left, the claim follows directly from the definition of $\De_n$. For the induction step, we consider the following cases:
\begin{rlist}

\item[\rm (a)]	 Suppose that $\eps$ is $s_1\lor s_2\le t$ and $\notmdl{\V}\eps$. Then, by property (i), $\notmdl{\V}s_1 \le t$ or $\notmdl{\V}s_2 \le t$, and, by the induction hypothesis, $\{s_1\le t\}\adm{\V}\De_n$ or $\{s_2\le t\}\adm{\V}\De_n$. By property (ii), $\{\eps\}\adm{\V}\{s_1\le t\}$ and $\{\eps\}\adm{\V}\{s_2\le t\}$, so also $\{\eps\}\adm{\V}\De_n$.

\item[\rm (b)]	 Suppose that $\eps$ is $s \le t_1\land t_2$ and $\notmdl{\V}\eps$. Then, by property (iv), $\notmdl{\V}s\le t_1$ or $\notmdl{\V}s\le t_2$, and, by the induction hypothesis, $\{s\le t_1\}\adm{\V}\De_n$ or $\{s\le t_2\}\adm{\V}\De_n$. By property (v), $\{\eps\}\adm{\V}s\le t_1$ and $\{\eps\}\adm{\V}s\le t_2$, so also $\{\eps\}\adm{\V}\De_n$.

\item[\rm (c)]	 Suppose that $\eps$ is $s_1\land s_2 \le t_1\lor t_2$ and $\notmdl{\V}\eps$. Then, by properties (iii) and (vi), $\notmdl{\V}s_1\le t_1\lor t_2$, $\notmdl{\V}s_2\le t_1\lor t_2$, $\notmdl{\V}s_1\land s_2\le t_1$, and  $\notmdl{\V}s_1\land s_2\le t_2$. So, by the induction hypothesis, $\{s_1\le t_1\lor t_2\}\adm{\V}\De_n$, $\{s_2\le t_1\lor t_2\}\adm{\V}\De_n$, $\{s_1\land s_2\le t_1\}\adm{\V}\De_n$, and $\{s_1\land s_2\le t_2\}\adm{\V}\De_n$. But also $\{\eps\}\adm{\V}\{s_1\le t_1\lor t_2,s_2\le t_1\lor t_2,s_1\land s_2\le t_1,s_1\land s_2\le t_2\}$, by Whitman's condition, so $\{\eps\}\adm{\V}\De_n$.

\item[\rm (d)]	 Suppose that $\eps$ is (up to a permutation of meets) $(s_1 \lor s_2) \land s_3\le y_i$ and $\notmdl{\V}\eps$. Then, by properties (vi) and (i), $\notmdl{\V} s_3 \le y_i$ and either $\notmdl{\V} s_1 \le y_i$ or $\notmdl{\V}  s_2 \le y_i$. Hence, by property (viii), either $\notmdl{\V}s_1\land s_3 \le y_i$ or  $\notmdl{\V}s_2\land s_3 \le y_i$, and, by the induction hypothesis, either $\{s_1\land s_3 \le y_i\}\adm{\V}\De_n$ or  $\{s_2\land s_3\le y_i\}\adm{\V}\De_n$. By the monotonicity of the lattice operations, $\{(s_1 \lor s_2) \land s_3\le y_i\}\mdl{\Lat}s_j\land s_3 \le y_i$ for $j\in\{1,2\}$, so also $\{\eps\}\adm{\V}\{s_j\land s_3 \le y_i\}$ for $j\in\{1,2\}$. Hence $\{\eps\}\adm{\V}\De_n$.

\item[\rm (e)]	 Suppose that $\eps$ is (up to a permutation of joins) $y_i\le (t_1 \land t_2) \lor t_3$ and $\notmdl{\V}\eps$. Then, by properties (iii) and (v), $\notmdl{\V} y_i\le t_3$ and either $\notmdl{\V} y_i\le t_1$ or $\notmdl{\V}y_i\le t_2$. Hence $\notmdl{\V}y_i\le t_1\lor t_3$ or  $\notmdl{\V}y_i\le t_2\lor t_3$, and, by the induction hypothesis, either $\{y_i\le t_1\lor t_3\}\adm{\V}\De_n$ or  $\{y_i\le t_2\lor t_3\}\adm{\V}\De_n$. By the monotonicity of the lattice operations, $\{y_i\le (t_1 \land t_2) \lor t_3\}\mdl{\Lat}y_i\le t_j\lor t_3$  for $j\in\{1,2\}$, so also $\{\eps\}\adm{\V}\{y_i\le t_j\lor t_3\}$ for $j\in\{1,2\}$. Hence $\{\eps\}\adm{\V}\De_n$. \qedhere

\end{rlist}
\end{proof}

\begin{corollary}
The terms $t_1,\dots,t_n \in \Tmc(\xbar)$ are $\Lat$-dependent if, and only if, for some $i \in \{1,\dots,n\}$,
\[
\mdl{\Lat} t_i \le\!\!\!\!\!\!  \bigvee_{j \in [n]{\setminus}\{i\}} \!\!\!\!\! t_j
\quad\text{or}\quad
\mdl{\Lat} \!\!\!\!\!\! \bigwedge_{j \in [n]{\setminus}\{i\}}\!\!\!\!\!\! t_j \le t_i.
\]
Hence dependence in the variety of lattices is decidable.
\end{corollary}


\section{Open Problems}\label{sec:openproblems}

We conclude with a short list of open problems:

\begin{rlist}

\item[(1)]
Proofs of uniform interpolation can be quite intricate. In particular, Pitts' constructive proof of this property for $\lgc{IPC}$ involves a complicated definition of left and right interpolants that is checked by induction on derivations in a terminating sequent calculus. It therefore makes sense to seek a simpler proof that $\lgc{IPC}$-dependence is decidable. Such a proof might perhaps use the fact that subalgebras of finitely generated free Heyting algebras are projective~\cite{Ghi99} and hence that finitely generated subalgebras of finitely generated free Heyting algebras are finitely presented. The challenge would then be to provide a simple algorithm for producing finite presentations of these subalgebras. Similarly, we expect that there should be a more direct proof of decidability for dependence in the variety of modal algebras that does not rely on the bisimulation-based method of~\cite{LW11} for calculating uniform interpolants.

\item[(2)]
It follows from Theorem~\ref{thm:lat} and Whitman's condition that the minimal $\DLat$-refuting sets of equations described in Example~\ref{exam:dlat} also serve as (non-minimal) $\Lat$-refuting sets of equations. This raises the question of whether it is the case for {\em any} variety of lattices $\V$ that $t_1,\dots,t_n \in \Tmc(\xbar)$ are $\V$-dependent if, and only if, for some $\emptyset \neq I \subsetneq [n]$,
\[
\mdl{\V} \bigwedge_{i \in I} t_i \le\!\!\!\bigvee_{j \in [n]{\setminus}I} \!\!\!\!t_j.
\]
Note that it is easily checked (since the free lattice on two generators is finite) that this is true for $n=2$; that is, in any variety $\V$ of lattices, $t_1,t_2$ are $\V$-dependent if, and only if, $\mdl{\V}t_1\le t_2$ or $\mdl{\V}t_2\le t_1$. 

\item[(3)]
For all the varieties considered in this paper, the dependence problem is decidable. It would therefore be interesting to know of an example (if one exists) of a variety with a decidable equational theory for which dependence is undecidable.

\item[(4)]
The decidability of dependence is an open problem for most non-locally finite varieties associated to non-classical logics. In particular, varieties of modal algebras for modal logics such as $\lgc{T}$, $\lgc{K4}$, $\lgc{S4}$, and $\lgc{KD}$ and varieties of pointed residuated lattices for  substructural logics such as $\lgc{FL}$ (the full Lambek calculus), $\lgc{MTL}$ (monoidal $t$-norm logic), $\lgc{R}$ (relevant logic), and $\lgc{MALL}$ (multiplicative additive linear logic), all fail to be coherent~\cites{KM18,KM18a} and it is not known if their finitely generated free algebras are coherent.

\end{rlist}


\bibliographystyle{model1a-num-names}


\end{document}